\documentclass[11pt,a4paper]{article}

\usepackage[british]{babel}
\usepackage[T1]{fontenc} 
\usepackage{amsmath}
\usepackage{amssymb} 
\usepackage{ucs} 
\usepackage[utf8x]{inputenc}

\usepackage[]{graphicx} 
\usepackage{latexsym}
\usepackage{amsthm} 
\usepackage{stmaryrd}
\usepackage{verbatim}
\usepackage{pgf,tikz}
\usetikzlibrary{arrows}

\usepackage{float}
\usepackage{hyperref}

\theoremstyle{definition}
\newtheorem{defi}{Definition}[section] 
\theoremstyle{plain}
\newtheorem{theo}[defi]{Theorem}
\newtheorem{theorem}[defi]{Theorem}
\newtheorem{coro}[defi]{Corollary} 

\newtheorem{lemma}[defi]{Lemma}
\newtheorem{prop}[defi]{Proposition}

\theoremstyle{remark}

\newcommand{\Ss}{\mathbb{S}}

\newcommand{\proj}{\mathbb{P}}

\newcommand{\R}{\mathbb{R}} 

\newcommand{\h}{\mathbb{H}}

\newcommand{\cC}{\mathcal{C}}

\newcommand{\PSL}{\mathrm{PSL}}

\newcommand{\AdS}{\mathrm{AdS}}
\newcommand{\ADS}{\mathbb{A}\mathrm{d}\mathbb{S}}

\newcommand{\SO}{\mathrm{SO}}
\newcommand{\PO}{\mathrm{PO}}
\newcommand{\rO}{\mathrm{O}}

\newcommand{\pscalb}[2]{\langle #1\, |\, #2 \rangle_{n,2} }
\newcommand{\tv}{\rightarrow}
\newcommand{\G}{\Gamma}
\newcommand{\g}{\gamma}
\newcommand{\RP}{\mathbb{RP}}

\DeclareMathOperator{\Id}{Id}

\newcommand{\Hyp}{\mathbb{H}}

\title{ Regularity of limit sets of  $\AdS$ quasi-Fuchsian groups} 
\author{Olivier Glorieux, Daniel Monclair}

\begin{document}
\maketitle

\setcounter{tocdepth}{1}

\begin{abstract}
Limit sets of  $\AdS$ quasi-Fuchsian groups of $\PO(n,2)$ are  always Lipschitz submanifolds. The aim of this article is to show that they are never $\mathcal{C}^1$, except for the case of Fuchsian groups. As a byproduct we show that $\AdS$ quasi-Fuchsian groups that are not Fuchsian are Zariski dense in $\PO(n,2)$. 
\end{abstract}

\section{Introduction}

 The study of various notions of convex cocompact groups in semi-simple Lie groups has gain considerable interest the last decade, thanks to its relation with Anosov representations. A particularly nice setting is for subgroups of $\PO(p,q)$ where the quadratic form helps to construct invariant domains of dicontinuity, see \cite{DGK}. 

In a previous paper, we studied the metric properties of limit sets for such representations \cite{glorieux2017hausdorff} and proved a rigidity result for quasi-Fuchsian representations in $\PO(2,2)$.  Recently Zimmer \cite{Zimmer} showed a $\mathcal{C}^2$ rigidity result for Hitchin representations in $\PSL_n(\R)$ ($\cC^\infty$ rigidity was known from the work of Potrie-Sambarino \cite{potrie2014}). 

In this paper, we study  the $\mathcal{C}^1$ regularity of such a limit set and prove a rigidity result for quasi-Fuchsian subgroups $\PO(n,2)$. They are examples of $\AdS$ convex-cocompact groups, as defined by \cite{DGK}.\\
\indent Given the standard quadratic form $q_{n,2}$ of signature $(n,2)$ on $\R^{n+2}$, we define $\ADS^{n+1}$ as the subset of $\R\proj^{n+1}$ consisting of negative lines for $q_{n,2}$. Its boundary $\partial\ADS^{n+1}$ is the set of $q_{n,2}$-isotropic lines.

The quadratic form $q_{n,2}$ induces a Lorentzian metric of constant sectional curvature $-1$ on $\ADS^{n+1}$.

\begin{defi}\cite{DGK}
A discrete subgroup $\G$ of $G=\PO(n,2) $ is $\AdS$ \emph{convex-cocompact} if it acts properly discontinously and cocompactly on some properly convex closed subset $\cC$ of $\ADS^{n+1}$ with nonempty interior whose ideal boundary $\partial_i \cC := \overline{\cC}\setminus \cC $ does not contain any nontrivial projective segment. 
\end{defi}

Any infinite convex-cocompact group contains proximal elements, ie. elements that have a unique attracting fixed point in $\partial \ADS^{n+1}$. For $\G$ a discrete subgroup of $\PO(n,2)$, the \emph{proximal limit set of $\G$ } is the closure $\Lambda_\G \subset \RP^{n,2}$  of the set of attracting fixed points of proximal elements of $\G$.
Since $\G$ acts properly discontinuously on a convex set $\cC$, the proximal limit set coincides with the ideal boundary of $\cC.$ It is  shown in \cite{DGK} that this notion of limit set coincides with the closure of  orbits in the boundary.

\begin{defi}
A discrete group of $\PO(n,2)$ is $\AdS$ \emph{quasi-Fuchsian} if it is $\AdS$ convex-cocompact and its proximal limit set is homeomorphic to a $n-1$ dimensional sphere. 

If moreover, the group preserves a totally geodesic copy of $\Hyp^n$ in $\ADS^{n+1}$, it is called \emph{$\AdS$-Fuchsian}. 
\end{defi}

The limit set of an $\AdS$-Fuchsian group is a geometric sphere, hence a  $\cC^1$-submanifold of $\partial \ADS$. 
The principal aim of this article is to show that the converse holds:

\begin{theorem}\label{theorem-main} Let $\G\subset \PO(n,2)$ be $\AdS$ quasi-Fuchsian. If $\Lambda_\G$ is a $\mathcal C^1$ submanifold of $\partial\ADS^{n+1}$, then $\G$ is Fuchsian.
\end{theorem}

The proof is based on the following result which is interesting on its own:

\begin{prop} \label{prop - fuchsien ou zariski dense intro}  
Let $\G\subset \PO(n,2)$ be $\AdS$ quasi-Fuchsian. If $\G$ is not $\AdS$-Fuchsian, then it is Zariski dense in $\PO(n,2)$.
\end{prop}

Remark that  this proposition and Zimmer's result  \cite[Corollary 1.48]{Zimmer} imply that the limit set is not $\mathcal{C}^2$.

\subsection*{Acknowledgements} Olivier Glorieux acknowledges support from the European Research Council (ERC) under the European Union’s Horizon 2020 research and innovation programme (ERC starting grant DiGGeS, grant agreement No 715982).

\section{Background on \texorpdfstring{$\AdS$}{AdS} quasi-Fuchsian groups. }
We introduce the results needed for the  proofs of Theorem \ref{theorem-main} and Proposition \ref{prop - fuchsien ou zariski dense intro}. Most of this section follows directly from the work of \cite{merigot2012anosov} and \cite{DGK}, except maybe the characterization of  Fuchsian groups as subgroups of $\rO(n,1)$ in Proposition \ref{prop-fuchsian est equivalent sous groupe de O(n,1)}.

First, let us define the anti-de Sitter space. We denote by $\pscalb{\cdot}{\cdot}$ the standard bilinear form of signature $(n,2)$ on $\R^{n+2}$, that is
\[ \pscalb{x}{y}=\sum_{i=1}^n x_iy_i - x_{n+1}y_{n+1}-x_{n+2}y_{n+2}~,\quad x,y\in \R^{n+2}. \]
\begin{defi}
The \emph{anti-de Sitter space} is defined by 
$$\ADS^{n+1} := \{ [x] \in \RP^{n+1} \, |\, \langle x | x \rangle_{n,2} < 0\}.$$
Its boundary is 
$$\partial \ADS^{n+1} := \{ [x] \in \RP^{n+1} \, |\, \langle x | x \rangle_{n,2} = 0\}.$$
Two points $[x],[y]\in\partial\ADS^{n+1}$ are called {\em transverse} if $\pscalb{x}{y}\neq 0$.
\end{defi}

We now give a brief review of the proximal limit set:
\begin{defi} Given $\g\in \PO(n,2)$, we denote by \[\lambda_1(\g)\geq\lambda_2(\g)\geq\cdots\geq\lambda_{n+2}(\g) \] the logarithms of the moduli of the eigenvalues of any of its representants in $\rO(n,2)$. We say that $\g$ is \emph{proximal} if $\lambda_1(\g)>\lambda_2(\g)$.
\end{defi}
Remark that an element of $\PO(n,2)$ has not always a lift in $\SO(n,2)$. However since it is the quotient of $\rO(n,2)$ by $\pm \Id$, the set of moduli of eigenvalues of a lift is well defined. If $\g\in\PO(n,2)$ is proximal, it has a unique lift $\hat\g\in\rO(n,2)$ which has $e^{\lambda_1(\g)}$ as an eigenvalue.

Notice that we always have $\lambda_3(\g)=\cdots=\lambda_n(\g)=0$, as well as $\lambda_1(\g)+\lambda_{n+2}(\g)=\lambda_2(\g)+\lambda_{n+1}(\g)=0$.

\begin{defi}
If $\g\in \PO(n,2)$ is proximal, we denote by $\g_+\in \R\proj^{n+1}$ its attracting fixed point, i.e.  the eigendirection for the eigenvalue of modulus $e^{\lambda_1(\g)}$ of a lift of $\g$ to $\rO(n,2)$. We also set $\g_-=(\g^{-1})_+$.
\end{defi}
Note that $\g_+$  is necessarily isotropic, i.e. $\g_+\in \partial \ADS^{n+1}$.


\begin{prop}[Proposition 5 in \cite{frances}]
If $\g \in \PO(n,2)$  is proximal, then $\lim_{n\to+\infty}\g^n(\xi)=\g_+$ for all $\xi\in\ADS^{n+1}\cup\partial\ADS^{n+1}$ which is transverse to $\g_-$ (i.e. such that $\pscalb{\xi}{\g_-}\neq 0$).
\end{prop}

Recall that the proximal limit set of a discrete subgroup  $\G\subset \PO(n,2)$ is the closure $\Lambda_\G$ in $\RP^{n+1}$ of the set of all attracting fixed points of proximal elements of $\G$, it is therefore a subset of $\partial \ADS^{n+1}$. 

If additionally $\G$ is $\AdS$ convex-cocompact, then it is word-hyperbolic and  the action of $\G$  on its proximal limit set is conjugated to the action on its Gromov boundary (see \cite[Theorem 1.11]{DGK} for the irreducible case, and \cite[Theorem 1.24]{DGK RPn} for the general case). As a consequence, we have: 
\begin{prop}\label{prop-action minimale sur l ensemble limite }\cite[Observation 27 p.153]{ghysdelaHarpe}
If $\G\subset \PO(n,2)$ is $\AdS$ convex-cocompact, then the action of $\G$ on the limit set $\Lambda_\G$ is minimal, i.e. all orbits are dense.
\end{prop}

The group $\rO(n,1)$ can be embedded in $\PO(n,2)$ by the following map: 
$$A \mapsto \left[\begin{array}{cc} 
A&0 \\
0 & 1\\
\end{array} \right].$$

We will say that an element (respectively a  subgroup) of $\PO(n,2)$ is conjugate to an element (respectively to  a  subgroup) of $\rO(n,1)$ if it has a conjugate in the image of this embedding.

The image of $\rO(n,1)$ is exactly the stabilizer of the point $[0:\cdots:0:1]\in \ADS^{n+1}$, so conjugates of $\rO(n,1)$ are exactly stabilizers of points in $\ADS^{n+1}$. Since a fixed point in $\ADS^{n+1}$ corresponds to an eigendirection which is negative for $\pscalb{\cdot}{\cdot}$, we immediately get the following characterization:
\begin{prop} 
\label{prop - element conjugue point fixe}
Let $\g\in \PO(n,2)$. Then $\g$ is conjugate to an element of $\rO(n,1)$ if and only if $\g$ fixes a point of $\ADS^{n+1}$. If $\g$ is proximal, this is also equivalent to $\lambda_2(\g)=0$.
\end{prop}
 A subgroup of $\PO(n,2)$ which is conjugate to a  cocompact lattice of $\rO(n,1)$ is  $\AdS$-Fuchsian, as it fixes a totally geodesic copy of $\Hyp^{n}$ on which it acts properly discontinuously and cocompactly. These are the only $\AdS$-Fuchsian groups: 
\begin{prop}\label{prop-fuchsian est equivalent sous groupe de O(n,1)}
A discrete group of $\PO(n,2)$ is $\AdS$-Fuchsian if and only if it is conjugate to a cocompact lattice of $\rO(n,1)$.
\end{prop}

\begin{proof}
Let $\G\subset \PO(n,2)$ be an $\AdS$-Fuchsian group. Let  $H$ be a totally geodesic copy of $\Hyp^n$ in $\ADS^{n+1}$ preserved by $\G$. Since the stabilizer $L\subset \PO(n,2)$ of $H$ is conjugate to $\rO(n,1)$, we only have to show that $\G$ is a cocompact lattice of $L$. This will be a consequence of the fact that $\G$ acts properly discontinuously and cocompactly on $H$.

 Let $\g$ be  a proximal element of $\G$. Let $\xi \in \partial H$ be transverse to the repelling fixed point  $\g_-$. The sequence $\g^n \xi $ lies in $\partial H$ and converges to $\g^+$. Therefore, $\partial H$ contains the attracting point of $\g$, and it follows that $\Lambda_\G \subset \partial H$. Since $\Lambda_\G$ and $\partial H$ are homeomorphic to $\Ss^{n-1}$, we have $\Lambda_\G =\partial H$. 

Finally since, $\G$ is convex-cocompact, $\G$ acts properly discontinuously and cocompactly on the convex hull of $\Lambda_\G$ that is $H$ \cite[Theorem 1.24]{DGK RPn}. 

\end{proof}

The boundary $\partial \ADS^{n+1}$ is naturally equipped with a conformal Lorentzian structure. It is conformally equivalent to the quotient of $\Ss^{n-1} \times \Ss^1 $ endowed with the Lorentzian conformal metric  $[g_{\Ss^{n-1}} - d\theta^2]$ (where $g_{\Ss^{n-1}} $ is the round metric of curvature $1$ on ${\Ss^{n-1}} $, and $d\theta^2$ is the round metric on the circle of radius one) by the antipodal map $(x,\theta)\mapsto(-x,-\theta)$. See \cite[paragraph 2.3]{merigot2012anosov} for more details.\\
\indent An important fact from pseudo-Riemannian geometry is that unparametrized isotropic geodesics (i.e. geodesics whose tangent vectors are isotropic) only depend on the conformal class. The isotropic geodesics of the metric $g_{\Ss^{n-1}}-d\theta^2$ on $\Ss^{n-1}\times \Ss^1$ are given by $\theta\mapsto (c(\theta),\theta)$ where $c:\Ss^1\to \Ss^{n-1}$ is a unit-speed geodesic. The image in $\partial \ADS^{n+1}$ of such an isotropic geodesic is the intersection of $\partial\ADS^{n+1}$ with a projective line.\\
\indent Using the absence of segments in the limit sets of $\AdS$ quasi-Fuchsian groups we have:

\begin{prop}\label{prop-l'ensemble limite est le graph d'une fonction lip}
The limit set $\Lambda_\G\subset\partial\ADS^{n+1}$ of an  $\AdS$ quasi-Fuchsian group $\G\subset\PO(n,2)$ is the quotient by the antipodal map of the graph of a distance-decreasing\footnote{That is $\forall x\neq y, \, d(f(x),f(y)) < d(x,y)$ .} map $f \, :\, \Ss^{n-1} \tv \Ss^{1}$ where $\Ss^{n-1}$ and $\Ss^{1}$ are endowed with the round metrics.
\end{prop}

\begin{proof}
Barbot and Mérigot showed in \cite{merigot2012anosov} that the limit set of a quasi-Fuchsian group lifts to  the graph of a 1-Lipschitz map. Since the limit set does not contain any non trivial segment of $\partial \ADS^{n+1}$ the map strictly decreases the distance. Indeed, if $d(f(x),f(y))=d(x,y)$ for some distinct $x,y\in \Ss^{n-1}$, then $f$ must be an isometry on the shortest geodesic segment from $x$ to $y$, and the graph of $f$ over this geodesic segment is a lightlike geodesic in $\Ss^{n-1}\times \Ss^1$, it thus descends to a straight line in $\partial\ADS^{n+1}$.
\end{proof}
Finally we will need the following proposition, which  in  the Lorentzian vocabulary translates as the fact that  the limit set is a Cauchy hypersurface:
\begin{prop}\label{prop-les geoedeisc de type lumiere intersect l ensemble limite}
If $\G\subset \PO(n,2)$ is $\AdS$ quasi-Fuchsian, then every isotropic geodesic of $\partial \ADS^{n+1} $  intersects  $\Lambda_\G$ at exactly one point. 
\end{prop}

\begin{proof}
Let $f \, :\, \Ss^{n-1} \tv \Ss^{1}$ be a distance-decreasing map such that the quotient by the antipodal map of its graph is $\Lambda_\G$. An isotropic geodesic can be parametrized by $\theta\mapsto(c(\theta), \theta)$, where $c : \Ss^1 \tv \Ss^{n-1} $ is a unit speed geodesic. Then the proposition is equivalent to the existence and uniqueness of a  fixed point for the map $f\circ c :\Ss^1\to\Ss^1$.

It is a simple exercise to show that a distance-decreasing map of a compact metric space to itself has a unique fixed point.  

\end{proof}

\section{The Zariski closure of \texorpdfstring{$\AdS$}{AdS} quasi-Fuchsian groups}
We prove in this section the Zariski density of $\AdS$ quasi-Fuchsian subgroups of $\PO(n,2)$ which are  not $\AdS$-Fuchsian. This result, which happens to be interesting in itself, will considerably simplify the proof of Theorem \ref{theorem-main} when we will use Benoist's Theorem \cite{benoist1997asymptotiques} about Jordan projections for discrete subgroups of semi-simple Lie groups in the last section.

\begin{lemma} \label{lem - fuchsien ou irreductible} Let $\G\subset \PO(n,2)$ be $\AdS$ quasi-Fuchsian. If $\G$ is reducible, then it is Fuchsian.
\end{lemma}
\begin{proof} Assume that $\G$ is not Fuchsian, and  let $V\subset \R^{n+2}$ be a $\G$-invariant subspace with $0<\dim(V)<n+2$.

 First, let us show that the restriction of $\pscalb{\cdot}{\cdot}$ to $V$ is non degenerate. Assume that it is not the case. Then $\G$  preserves the totally isotropic space $V\cap V^\perp$. It has dimension $1$ or $2$. If $\dim(V\cap V^\perp)=1$, then  $\proj(V\cap V^\perp)$ is a global fixed point for the action of $\G$ on  $\partial\ADS^{n+1}$, which cannot exist. The case $\dim(V\cap V^\perp)=2$ is impossible because it also implies the existence of a global fixed point  on  $\partial\ADS^{n+1}$ (the intersection of the null geodesic $\proj(V\cap V^\perp)$ of $\partial\ADS^{n+1}$ with $\Lambda_\G$ given by Proposition \ref{prop-les geoedeisc de type lumiere intersect l ensemble limite}).
 
 Thus the restriction of $\pscalb{\cdot}{\cdot}$ to $V$ is  non degenerate. It can have signature $(k,2)$, $(k,1)$ or $(k,0)$ (where $k\geq 0$ is the number of positive signs).
 
 In the first case, $\G$ acts on some totally geodesic copy $X$ of $\ADS^{k+1}$ (with $k<n$) in $\ADS^{n+1}$ (defined as $X=\proj(V)\cap\ADS^{n+1}$). Then $\partial X\cap \Lambda_\G$ is a non empty closed invariant subset of $\Lambda_\G$, hence $\Lambda_\G\subset \partial X$ and $C(\Lambda_\G)\subset X$. Since $C(\Lambda_\G)$ has non empty interior in $\ADS^{n+1}$ (Lemma 3.13 in \cite{merigot2012anosov}), we see that $X=\ADS^{n+1}$, i.e. $V=\R^{n+2}$, which is absurd.

Now assume that $V$ has Lorentzian signature $(k,1)$. Then $\G$ preserves $X=\proj(V)\cap \ADS^{n+1}$ which is a totally geodesic copy of $\h^k$. It also acts on $X'=\proj(V^\perp)\cap \ADS^{n+1}$ which is a totally geodesic copy of $\h^{k'}$ (with $k+k'=n$). Considering a proximal element $\g\in \G$, there is a point in  $\partial X\cup \partial X'$ which is transverse to the repelling fixed point $\g_-$ of $\g$ (otherwise $\g_-$ would be in $V\cap V^\perp$). This implies that $\g_+\in \partial X\cup\partial X'$, hence $\Lambda_\G\cap \partial X\neq \emptyset$ or $\Lambda_\G\cap \partial X'\neq \emptyset$. The action of $\G$ on $\Lambda_\G$ being minimal, we find that $\Lambda_\G\subset \partial X$ or $\Lambda_\G\subset \partial X'$. This is impossible because $\Lambda_\G$ is homeomorphic to $\Ss^{n-1}$ and $\partial X$ (resp. $\partial X'$) is homeomorphic to $\Ss^{k-1}$ (resp. $\Ss^{k'-1}$).

 Finally, if $V$ is  positive definite, then $V^\perp$ has signature $(n-k,2)$, this case has already been ruled out. 
\end{proof}

\begin{coro} \label{coro - composante neutre irreductible} If a subgroup $\G\subset \PO(n,2)$ is $\AdS$ quasi-Fuchsian but not $\AdS$-Fuchsian, then the identity component of the Zariski closure of $\G$ acts irreducibly on $\R^{n+2}$.
\end{coro}
\begin{proof} Let $\G_\circ\subset G$ be a finite index subgroup. Since $\Lambda_{\G_\circ}=\Lambda_\G$, it cannot be Fuchsian, so it acts irreducibly on $\R^{n+2}$ by Lemma \ref{lem - fuchsien ou irreductible}.
\end{proof}

\begin{prop} \label{prop - fuchsien ou zariski dense} Let $\G\subset \PO(n,2)$ be $\AdS$ quasi-Fuchsian. If $\G$ is not $\AdS$-Fuchsian, then it is Zariski dense in $\PO(n,2)$.
\end{prop}
\begin{proof} Let $G\subset \SO_0(n,2)$ be the pre-image by the quotient map $\SO_0(n,2)\to\PO(n,2)$ of the identity component  of the Zariski closure of $\G$, and assume that $\G$ is not Fuchsian.

By Corollary \ref{coro - composante neutre irreductible}, we know that $G$ acts irreducibly on $\R^{n+2}$. According to \cite[Theorem 1]{discala} the only connected irreducible subgroups of $\SO(n,2)$ other than $\SO_0(n,2)$ are  $\mathrm{U}(\frac{n}{2},1)$, $\mathrm{SU}(\frac{n}{2},1)$, $\Ss^1.\SO_0(\frac{n}{2},1)$ (when $n$ is even) and $\SO_0(2,1)$ (when $n=3$).
 
 The first three cases are subgroups of   $\mathrm{U}(\frac{n}{2},1)$, which only contains elements $\g\in \SO(n,2)$ satisfying $\lambda_1(\g)=\lambda_2(\g)$ so $G$ cannot be one of them (otherwise $\G$ would not contain any proximal element and $\Lambda_\G=\emptyset$).
 
 The irreducible copy of $\SO_0(2,1)$ in $\SO(3,2)$ can also be ruled out because a quasi-Fuchsian subgroup of $\PO(3,2)$ has cohomological dimension $3$, so it cannot be isomorphic to a discrete subgroup of $\SO_0(2,1)\approx \PSL(2,\R)$.
 
 The only possibility left is that $\G$ is Zariski dense in $\PO(n,2)$.

\end{proof}

\section{Non differentiability of limit sets} 
We finally prove the main result, Theorem \ref{theorem-main}. The proof goes as follows: first, we prove that the tangent spaces of the limit set are space like (i.e. positive definite for the natural Lorentzian conformal structure on $\partial \ADS^{n+1}$).  Then by an algebraic argument, this shows that all proximal elements of $\G$ are conjugate (by an \textit{a priori} different element of $\PO(n,2)$) to an element of $\rO(n,1)$. Finally, using a famous theorem of Benoist, this implies that $\G$ is not Zariski-dense, and therefore by Proposition \ref{prop - fuchsien ou zariski dense intro} that the group is Fuchsian.

\subsection{Spacelike points}
\begin{lemma} \label{lem - point spacelike} If $\G\subset \PO(n,2)$ is $\AdS$ quasi-Fuchsian and $\Lambda_\G$ is a $\mathcal C^1$ submanifold of $\partial\ADS^{n+1}$, then there is $\xi\in \Lambda_\G$ such that $T_\xi\Lambda_\G$ is spacelike.
\end{lemma}

\begin{proof} Let $f \, :\, \Ss^{n-1} \tv \Ss^{1}$ be a distance-decreasing map such that the quotient by the antipodal map of its graph is $\Lambda_\G$.

Knowing that the graph of $f$ is a $\cC^1$-submanifold, we first want to show that $f$ is $\cC^1$. Using the Implicit Function Theorem, it is enough to know that the tangent space of the graph projects non trivially to the tangent space of $\Ss^{n-1}$. This is true because $f$ is Lipschitz.

Since $f$ satisfies $d(f(x),f(y))<d(x,y)$ for $x\neq y$ (Proposition \ref{prop-l'ensemble limite est le graph d'une fonction lip}), it cannot be onto, so it can be seen as a function $f:\Ss^{n-1}\to \R$. At a point $x\in\Ss^{n-1}$ where it reaches its maximum, it satisfies $df_x=0$, so the tangent space to $\Lambda_\G$ at $(x,f(x))$ is $T_x\Ss^{n-1}\times \{0\}$, which is spacelike.
\end{proof}

\begin{coro} \label{coro - partout spacelike}If $\G\subset \PO(n,2)$ is $\AdS$ quasi-Fuchsian and $\Lambda_\G$ is a $\mathcal C^1$ submanifold of $\partial\ADS^{n+1}$, then for all $\xi\in \Lambda_\G$, the tangent space $T_\xi\Lambda_\G$ is spacelike.

\end{coro}

\begin{proof} Let $E:=\{\xi\in\Lambda_\G : T_\xi\Lambda_\G\textrm{ is spacelike} \}$. Then $E$ is open and $\G$-invariant. Since the action of $\G$ on $\Lambda_\G$ is conjugate to the action on its Gromov boundary, it is minimal (i.e. all orbits are dense). It follows that $E$ is either empty or equal to $\Lambda_\G$ and by Lemma \ref{lem - point spacelike}, it is not empty.
\end{proof}

\paragraph{Remark:} Lemma \ref{lem - point spacelike}   fails in general in higher rank pseudo-Riemannian symmetric spaces, i.e. for $\h^{p,q}$ quasi-Fuchsian groups. Indeed,  Hitchin representations  in $\PO(2,3)$ provide $\h^{2,2}$ quasi-Fuchsian groups which are not $\h^{2,2}$-Fuchsian, yet have a $\cC^1$ limit set (which is isotropic for the natural Lorentzian conformal structure on $\partial\h^{2,2}$). 


\subsection{Fixed points and Benoist's asymptotic cone}
\begin{lemma} \label{lem - conjugaison elements} Let $\G\subset \PO(n,2)$ be $\AdS$ quasi-Fuchsian. If the limit set $\Lambda_\G\subset \partial \ADS^{n+1}$ is a $\cC^1$ submanifold, then every proximal element $\g\in\G$ is conjugate in $\PO(n,2)$ to an element of $\mathrm O(n,1)$.
\end{lemma}
\begin{proof} Let $\g\in\G$ be proximal, and let $\hat\g\in \rO(n,2)$ be the lift with eigenvalue $e^{\lambda_1(\g)}$. Let $\g_+\in\Lambda_\G$ be the attracting fixed point. Then the differential of $\g$ acting on $\partial\ADS^{n+1}$ at $\g_+$ preserves $T_{\g_+}\Lambda_\G$. It also preserves $(T_{\g_+}\Lambda_\G)^\perp$, which is a timelike line because of Corollary \ref{coro - partout spacelike}.

Lifting everything to $\R^{n+2}$ and using the identification of $T_{\g_+}\partial\ADS^{n+1}$ with $\g_+^\perp/\g_+$, we see that $\hat\g$ preserves a two-dimensional plane $V\subset \g_+^\perp$ which contains $\g_+$ and a negative direction. Let $(u,v)$ be a basis of $V$, where $u\in\g_+$ and $\pscalb{v}{v}=-1$.

By writing $\hat\gamma v = au+bv$, we find that $b^2=-\pscalb{\hat\gamma v}{\hat\gamma v}=-\pscalb{v}{v}=1$. So the matrix of the restriction of $\hat\g$ to $P$ in the basis $(u,v)$ has the form
\[\begin{pmatrix}
e^{\lambda_1(\g)} & a\\
0 & \pm 1
\end{pmatrix}\]

It has $\pm 1$ as an eigenvalue, and the eigendirection is in $V$ but is not $\g_+$ (because $\lambda_1(\g)>0$), so it is negative for $\pscalb{\cdot}{\cdot}$. This eigendirection is a point of $\ADS^{n+1}$ fixed by  $\g$, and Proposition \ref{prop - element conjugue point fixe} implies that $\g$ is conjugate to an element of $\rO(n,1)$.
\end{proof}

\begin{theo} Let $\G\subset \PO(n,2)$ be $\AdS$ quasi-Fuchsian. If the limit set $\Lambda_\G\subset \partial \ADS^{n+1}$ is a $\cC^1$ submanifold, then $\G$ is Fuchsian.
\end{theo}

\begin{proof} By Lemma \ref{lem - conjugaison elements}, the Jordan projections of proximal elements of  $\G$ all lie in a half line in a  Weyl chamber $\mathfrak a^+$ of $\PO(n,2)$, therefore its asymptotic cone has empty interior in $\mathfrak a^+$. Benoist's Theorem \cite{benoist1997asymptotiques} implies $\G$ is not Zariski dense. Proposition \ref{prop - fuchsien ou zariski dense intro} implies that $\G$ is Fuchsian.
\end{proof}

~\\
\footnotesize \textsc{}\\
 \emph{E-mail address:}  \verb|olivier.glrx@gmail.com|
~\\
\footnotesize \textsc{Lycée Chaptal, 45 Bd des Batignolles, 75008 Paris, France}\\
 \emph{E-mail address:}  \verb|daniel.monclair@universite-paris-saclay.fr|
 ~\\
 \footnotesize \textsc{Institut de Mathématique d'Orsay, Bâtiment 307, Université Paris-Saclay, F-91405 Orsay Cedex, France}\\


\begin{thebibliography}{Glo15b}







\bibitem[Bar15]{barbot2015deformations}
Thierry Barbot.
\newblock Deformations of Fuchsian AdS representations are quasi-Fuchsian.
\newblock {\em J. Differential Geom.}, \textbf{101} (2015), no. 1,  1--46.


\bibitem[Ben97]{benoist1997asymptotiques}
Yves Benoist.
\newblock Propriétés asymptotiques des groupes linéaires.
\newblock{\em Geom.  Funct. Anal.}, \textbf{7}  (1997),  no. 1, 1–47.


\bibitem[BBZ07]{barbot2007constant}
Thierry Barbot, Fran{\c{c}}ois B{\'e}guin, and Abdelghani Zeghib.
\newblock Constant mean curvature foliations of globally hyperbolic spacetimes
  locally modelled on $\AdS_3$.
\newblock {\em Geom. Dedicata}, {\bf 126} (2007), no. 1, 71--129.

\bibitem[BM12]{merigot2012anosov}
 Thierry Barbot and Quentin M{\'e}rigot.
\newblock Anosov AdS representations are quasi-fuchsian.
\newblock {\em Groups Geom. Dyn.}, {\bf 6} (2012), no. 3, 441--483.


%
%
%
%
%
%
%
%
%
%
%
%



\bibitem[DGK17]{DGK RPn}
Jeffrey Dancinger, Fanny Kassel, Francois Guéritaud.
\newblock {Convex cocompact actions in real projective geometry.}
\newblock To appear in {\em Ann. Sci. \'Ec. Norm. Supér.} 

\bibitem[DGK18]{DGK}
Jeffrey Dancinger, Fanny Kassel, Francois Guéritaud.
\newblock {Convex cocompactness in pseudo-Riemannian hyperbolic spaces.}
\newblock {\em Geom. Dedicata }, {\bf 192} (2018),  87--126.





\bibitem[DSL]{discala}
 Antonio J. Di Scala and Thomas Leistner.
 \newblock{Connected subgroups of $\SO(2, n)$ acting irreducibly on $\R^{2,n}$.}
\newblock{{\em Isr. J. Math.}, {\bf 182} (2011), 103--121.}
%

\bibitem[Fra05]{frances}
Charles Frances.
\newblock Lorentzian Kleinian groups. 
\newblock{ {\em Comment. Math. Helv.}, {\bf 80} (2005), no. 4, 883--910.  }



\bibitem[GdlH]{ghysdelaHarpe}
Etienne Ghys, Pierre de la Harpe.
\newblock Sur les groupes Hyperboliques d'après Mikhael Gromov.
\newblock {\em Progr. Math.}, {\bf 83},  {\em Birkhäuser Boston, Inc., Boston, MA,}, 1990. xii + 285 pp.
%
\bibitem[GM]{glorieux2017hausdorff}
Olivier Glorieux, Daniel Monclair.
\newblock Critical exponent and Hausdorff dimension in pseudo-Riemannian hyperbolic geometry
\newblock {\em Int. Math. Res. Not. IMRN} (2021), no. 18, 13661--13729. 



%

%

\bibitem[KK16]{Kassel2016}
Fanny Kassel and Toshiyuki Kobayashi.
\newblock Poincaré series for non-Riemannian locally symmetric spaces. 
\newblock In {\em  Adv. Math.}, {\bf 287} (2016),  123--236.
%
%
%

\bibitem[Lab06]{labourie2006anosov}
Fran{\c{c}}ois Labourie.
\newblock Anosov flows, surface groups and curves in projective space.
\newblock {\em Invent. Math.}, {\bf 165} (2006), no.1, 51--114.

%
%
%
%
%
%
%

\bibitem[Mes07]{mess2007lorentz}
Geoffrey Mess.
\newblock Lorentz spacetimes of constant curvature.
\newblock {\em Geom. Dedicata}, {\bf 126} (2007),  3--45.
%
%
%
%
%
%

\bibitem[PS17]{potrie2014}
Raphael Potrie and Andres Sambarino. 
\newblock Eigenvalues and entropy of a Hitchin representation.
\newblock {\em Invent. Math.}, {\bf 209} (2017), no. 3, 885--925.

%
%
%
%
\bibitem[Zim18]{Zimmer}
Andrew Zimmer.
\newblock Projective Anosov representations, convex cocompact actions, and rigidity. 
\newblock {\em J. Differential Geom.}, {\bf 119} (2021), no.3, 513--586.

\end{thebibliography}
\end{document}